\documentclass[11pt]{amsart}

\usepackage{epigamath}
\usepackage{microtype}

\usepackage[english]{babel}

\numberwithin{equation}{section}

\usepackage{mathtools}

\usepackage{framed}
\usepackage{color}

\usepackage{mathrsfs}
\usepackage{cite}
\usepackage{enumerate}
\usepackage{caption}

\usepackage[all,2cell]{xy} \SilentMatrices
\usepackage{epsfig,color}

\theoremstyle{plain}
\newtheorem{theorem}{Theorem}[section]
\newtheorem*{lemma*}{Lemma}
\newtheorem{lemma}[theorem]{Lemma}
\newtheorem*{theorem*}{Theorem}
\newtheorem{proposition}[theorem]{Proposition}
\newtheorem*{proposition*}{Proposition}
\newtheorem{corollary}[theorem]{Corollary}
\newtheorem*{corollary*}{Corollary}

\newtheorem*{thmA}{Theorem A}
\newtheorem*{thmB}{Theorem B}

\theoremstyle{definition}
\newtheorem{remark}[theorem]{Remark}
\newtheorem*{remark*}{Remark}
\newtheorem*{definition*}{Definition}

\newtheorem*{example*}{Example}

\newtheorem{question}[theorem]{Question}
\newtheorem*{question*}{Question}

\def\H{\operatorname{H}}

\def\c1{\operatorname{c_1}}
\def\c2{\operatorname{c_2}}

\def\ZZ{{\mathbb Z}}

\def\PP{{\mathbb P}}

\def\O{{\mathcal O}}

\def\+{\oplus}                   
\def\*{\otimes}

\def\Gr{\operatorname{Gr}}

\def\mov{\overline{\operatorname{Mov}}}

\def\eff{\overline{\operatorname{Eff}}}
\def\nef{\operatorname{Nef}}

\def\Pic{\operatorname{Pic}}

\def\Sym{\operatorname{S}}

\def\tH{\widetilde{H}}
\def\HC{\operatorname{HC}}
\def\HH{\textrm{H}}

\EpigaVolumeYear{4}{2020}
\EpigaArticleNr{8}
\ReceivedOn{November 28,
2019}
\InFinalFormOn{March 14, 2020}
\AcceptedOn{April 22, 2020}

\title{Remarks on the positivity of the cotangent bundle of a K3 surface}
\titlemark{Positivity of the cotangent bundle of a K3}

\author{Frank Gounelas}
\address{TU M\"unchen, Zentrum Mathematik - M11, Boltzmannstr. 3, 85748 Garching M\"unchen, Germany}
\email{gounelas@ma.tum.de}

\author{John Christian Ottem}
\address{Department of Mathematics, University of Oslo, Box 1053, Blindern, 0316 Oslo, Norway}
\email{johnco@math.uio.no}

\authormark{F. Gounelas and J. C. Ottem}

\AbstractInEnglish{Using recent results of Bayer--Macr\`i, we compute in many cases the pseudoeffective and nef cones of the projectivised
cotangent bundle of a smooth projective K3 surface. We then use these results to construct explicit families of smooth curves on which the
restriction of the cotangent bundle is not semistable (and hence not nef). In particular, this leads to a counterexample
to a question of Campana--Peternell.}



\TitleInFrench{Remarques sur la positivit\'e du fibr\'e cotangent d'une surface K3}

\AbstractInFrench{En utilisant des r\'esultats r\'ecents de Bayer--Macr\`i, nous calculons dans de nombreux cas le c\^one pseudoeffectif et le c\^one nef du projectivis\'e du fibr\'e cotangent d'une surface K3 projective. Nous utilisons ensuite ces r\'esultats pour construire des familles explicites de courbes lisses sur lesquelles la restriction du fibr\'e cotangent n'est pas semi-stable (et n'est donc pas nef). En particulier, cela fournit une r\'eponse n\'egative \`a une question de Campana--Peternell.}

\acknowledgement{The first author acknowledges the support of the ERC Consolidator Grant 681838 K3CRYSTAL.}

\begin{document}



\maketitle

\begin{prelims}

\DisplayAbstractInEnglish

\bigskip

\DisplayKeyWords

\medskip

\DisplayMSCclass

\bigskip

\languagesection{French}

\bigskip

\DisplayTitleInFrench

\medskip

\DisplayAbstractInFrench

\end{prelims}


\newpage

\setcounter{tocdepth}{2}

\tableofcontents


\section{Introduction}

Miyaoka proved that the cotangent bundle of a non-uniruled variety is {generically nef}, in the sense that its restriction to a sufficiently ample and general complete intersection curve is a nef vector bundle \cite{mehtaramanathan}. This in turn has many interesting geometric consequences; see \cite{campanapeternell,peternell2011generically} for more general properties of such vector bundles, in particular for the tangent and cotangent bundles. The starting point of the present paper was the following question of Campana and Peternell:

\begin{question}[Question 1.6 in \cite{campanapeternell}]
Let $f:X\to Y$ be a birational morphism of smooth projective non-uniruled varieties. Is $f^*\Omega^1_Y$ generically nef?
\end{question}

We answer this question in the negative, by constructing a certain blow-up $f:X\to S$ of a K3 surface. Let us sketch the construction. Let $S$ be a generic quartic surface and let $\PP(\Omega^1_S)$ denote the projectivisation of the cotangent bundle. Let $L=\O_{\PP(\Omega_S^1)}(1)$ be the relative hyperplane bundle and $H$ the pullback of the polarisation on $S$. Then one finds that $D=L+2H$ is a base-point free ample divisor (see Proposition \ref{coneprojspace}), so a generic element $X\in |D|$ is a smooth surface and the projection $\pi:X\to S$ is birational. Moreover, as $L\cdot D^2<0$, the canonical quotient $\pi^*\Omega^1_S\to L\to 0$, shows that $\pi^*\Omega^1_S$ is not nef restricted to any smooth curve $C\subset X$ in the (ample) linear system $|mD|_X|$. In particular $\pi^*\Omega^1_S$ is not generically nef.

This example has other interesting features. For instance, the curve $C$ maps isomorphically to a (movable) smooth curve $D\subset S$ with the property that $\Omega^1_S|_D$ is not nef, and hence not stable. This addresses \cite[Example 4.3]{campanapaun}, which asks about explicit examples of curves destabilising the cotangent bundle of a K3. We show more generally the following:

\begin{thmA}
	Let $(S,\O_S(1))$ be a general polarised K3 surface of degree $d$. 
    \begin{enumerate}[$(i)$]
    	\item If $d=2$ then there is a positive-dimensional family of smooth curves $C$ in $|\O_S(6)|$ so that $\Omega^1_S|_C$ is not semistable.
        \item If $d=4,6$ there are families as above in $|\O_S(3)|$ $($and another in $|\O_S(4)|$ if $d=4)$.
        \item If $d=8$ there is a family in $|\O_S(3)|$ so that $\Omega^1_S|_C$ is strictly semistable.
    \end{enumerate}
\end{thmA}

These results should be compared with the results of Bogomolov and Hein (see Theorem \ref{hein}), which essentially say that if either the degree of the polarisation or the multiple of the polarisation in which the curve lies is high, then there can be no smooth destabilising curves, and also that a general curve in any multiple of the polarisation is not destabilising.

More generally, one can ask how the sets of all curves and effective divisors on $\PP(\Omega^1_S)$ reflect the geometric
properties of $S$. In particular this is closely related to describing the nef and effective cones on $\PP(\Omega^1_S)$,
and in turn the existence of sections of the twisted symmetric differentials $\Sym^a(\Omega^1_S)\otimes \O_S(b)$. These
subjects have a long history, going back to the work of Kobayashi\cite{kobayashi}, who showed that a simply connected
Calabi--Yau variety has no symmetric differentials, i.e., $\HH^0(X,\Sym^m\Omega^1_X)=0$ for all $m\ge 1$. For K3 surfaces, this was extended by Nakayama \cite{nakayama}, who showed that in fact $L=\O_{\PP(\Omega^1_S)}(1)$ is not even {pseudoeffective} on $\PP(\Omega^1_S)$ (i.e., its class is not a limit of effective classes). Very recently this was extended to simply connected Calabi--Yau threefolds by Druel \cite[Theorem 6.1]{druelcy3} and by H\"oring--Peternell \cite{hoeringpeternell} to all dimensions.

In general computing these cones explicitly seems like a difficult problem. In some low-degree cases one can use the well-known projective models to study the cohomology of the symmetric powers of the cotangent bundle, but in general the cones seem to depend on the degree $d$ in a rather subtle way. We are able to solve the problem at least for infinitely many $d$, using results of Bayer--Macr\`i.

The starting observation is that $\PP(\Omega^1_S)$ embeds in the Hilbert scheme $S^{[2]}$ as the exceptional divisor of
the Hilbert--Chow morphism, so one can try restricting extremal divisors from the pseudoeffective and nef cones of
$S^{[2]}$ to $\PP(\Omega^1_S)$, which are known from the results of Bayer--Macr\`i \cite{bayermacri}: the effective cone
$\eff(S^{[2]})$ (resp.\ nef cone) is spanned by $B$ (resp.\ $\widetilde{H}$) and one more extremal ray which in what follows
we call the \textit{second extremal ray} (see Theorem \ref{bm}). We prove that
the restriction of this second ray is indeed extremal for infinitely many $d$, but also that this is not always the case. In what
follows, we will consider the Pell-type equation \begin{equation}\label{Pell}x^2-4ty^2 = 5.\end{equation}

\begin{thmB}[See Section \ref{sectioncones}] Let $(S,\O_S(1))$ be a primitively polarised K3 surface of  degree $d=2t$ and Picard number one.
First, if $t$ is a square and \eqref{Pell} has no solution, then 
$$ \eff(\PP(\Omega^1_S)) = \nef(\PP(\Omega^1_S)) = \left\langle H, L+\frac{2}{\sqrt{t}}H\right\rangle.$$
Next, the restriction of the second extremal ray of $\eff(S^{[2]})$ is extremal on $\eff(\PP(\Omega^1_S))$ if the following three conditions hold
\begin{enumerate}[$(i)$]
	\item $t$ is not a square, 
	\item \eqref{Pell} has no solutions, 
	\item the minimal solution $(a,b)$ of $x^2-ty^2=1$ has $b$ even.
\end{enumerate}
Finally, the restriction of the second extremal ray of $\nef(S^{[2]})$ is extremal on $\nef(\PP(\Omega^1_S))$ if \eqref{Pell} has a solution.
\end{thmB}In particular we know the effective cone of $\PP(\Omega^1_S)$ if 
$$d=4,8,12,18,20,24,26,32,34,36,\ldots,$$
whereas we know the nef cone if 
$$d=2,8,10,18,22,32,38,50,58,62,\ldots.$$
We emphasise that in the above cases the cones are explicitly computable, the slopes given in terms of solutions to Pell-type equations. We defer to Section \ref{sectioncones} for particulars. 

As mentioned above, we also prove that the above trick does not work in general: we show that the restriction of the nef
cone of $S^{[2]}$ is a strictly smaller subcone of $\nef(\PP(\Omega^1_S))$ in the degree four and six cases (see Section
\ref{sectionexamples}). The cases in which this restriction is extremal depend on the geometry of the minimal models of
$S^{[2]}$, and one can find results and a classification in work of Bayer--Macr\`i, Hassett--Tschinkel, Markman and others.

In the final section we analyse in more detail these contractions in the cases of K3s of degree $d=2,4,6,8$, improving bounds on, or computing the aforementioned slopes, and describe the geometry of the extremal divisors.
 
\vskip 0,2cm
\textbf{Notation.} We work over the complex numbers. For $G$ a vector bundle on a variety $X$ we say that $G$ is pseudoeffective, big, nef, ample  if $\O_{\PP(G)}(1)$ is a pseudoeffective, big, nef, ample line bundle on $\PP(G)$ respectively. We use Grothendieck notation throughout, so that $\PP(G)$ parameterises one-dimensional quotients of $G$. In particular if $\pi:\PP(G)\to X$ is the projection, we have a universal quotient line bundle $\pi^*G\to\O_{\PP(G)}(1)\to 0$, and for a surjection $G_1\to G_2\to0$ we have an induced inclusion $\PP(G_2)\subseteq\PP(G_1)$ so that $\O_{\PP(G_1)}(1)|_{\PP(G_2)}=\O_{\PP(G_2)}(1)$.

\vskip 0,2cm
\textbf{Acknowledgements.} We would like to thank Yohan Brunebarbe and Mike Roth for numerous
conversations which got this project started but also N. Addington, B. Bakker, J. Gu\'er\'e, B.
Hassett, H.-Y. Lin, E.  Macr\`i, M. Mauri, G. Mongardi, K. O'Grady, J. Sawon, R. Vacca for helpful
discussions. We would also like to thank the Mathematisches Forschungsinstitut of Oberwolfach for
providing excellent working conditions when this project began during a \emph{Research in Pairs} of the
authors.

\section{Preliminaries}\label{prelim}
\noindent Throughout the paper we will let $S$ denote a K3 surface with a primitive ample line bundle $\O_S(1)$ of
degree $d=\O_S(1)^2$, which is always an even number, so we will sometimes write $d=2t$. We will mostly assume that $\Pic(S)=\ZZ$, generated by $\O_S(1)$. Let $E=\PP(\Omega^1_S)$ and $\pi:E\to S$ the projection. In this case the Picard group of $E$ is generated by $L=\O_{\PP(\Omega^1_S)}(1)$ and $H=\pi^*\O_S(1)$. We have the following intersection numbers:
\begin{alignat*}{3}
L^3=s_2(\Omega^1_S)&=-24,\qquad && L\cdot H^2&=d, \\
L^2\cdot H=s_1(\Omega^1_S)&=0,\qquad  && \quad H^3&=0.
\end{alignat*}
We denote by $\alpha_{e,d}$ and $\alpha_{n,d}$ the (positive) real numbers which are the slopes of the pseudoeffective and nef cones of $E$ respectively. In other words, \begin{align*}
\eff(\PP(\Omega^1_S))=\left\langle H, L+\alpha_{e,d}H\right\rangle& \\
\nef(\PP(\Omega^1_S))=\left\langle H, L+\alpha_{n,d}H\right\rangle&.
\end{align*}
These numbers are often called the \textit{pseudoeffective} and \textit{nef thresholds} for $\Omega^1_S$. We denote by $S^{[2]}$ the Hilbert scheme of length two subschemes on $S$. For a divisor $H$ on $S$, we get a divisor $\tH$ on $S^{[2]}$ by taking subschemes incident to $H$. We also have the Hilbert--Chow morphism 
$$\HC:S^{[2]}\to \operatorname{Sym}^2S$$taking a subscheme to its support. This gives $\Pic(S^{[2]})\simeq \Pic(S)\oplus \ZZ B$ where $B=\frac{1}{2}E$ and $E$ is the exceptional divisor of $\HC$ (the divisor of non-reduced subschemes). Note in particular that $E\simeq\PP(\Omega^1_S)$. In terms of $L$ and $H$, we have the restrictions 
\begin{equation*}
\tH|_E = 2H, \qquad B|_E=-L,
\end{equation*}
where the first follows from the fact that $\operatorname{pr}_1^*\O_S(1) \otimes \operatorname{pr}_2^*\O_S(1)$ restricts
to $\O_S(2)$ on the diagonal $\Delta\subset S\times S$.

\subsection{Stability of restrictions}\label{sectiondestabcurves}  

As mentioned in the introduction, the positivity of $\Omega^1_S$ on a curve is closely related to stability. This is because of Hartshorne's theorem \cite[\S 6.4.B]{laz2} which implies that on a smooth curve, a vector bundle of degree zero is nef if and only if it is semistable. In particular,
\begin{lemma}\label{lazlemma}
	Let $(S,\O_S(1))$ be a K3 surface and $C\subset S$ be a smooth curve. Then $\Omega^1_S|_C$ is nef if and only if it is semistable.
\end{lemma}

A consequence of Yau's Theorem is that $\Omega^1_S$ is stable with respect to any polarisation. This does not imply that the restriction to each curve $C$ is stable, as there can be negative quotients $\Omega^1_S|_C\to Q\to 0$ which do not come from quotients of $\Omega^1_S$. On the other hand, one can apply results of Bogomolov to show that for any stable vector bundle $G$ on a surface $S$, if we take a smooth curve in a (computable) high enough multiple of the polarisation, then the restriction $G|_C$ is stable. In our case, Bogomolov's theorem \cite[Theorem 7.3.5]{huybrechtslehn} implies that if $C\in|\O_S(n)|$ for $n>49$ then $\Omega^1_S|_C$ is stable. In fact, if $S$ is a K3 surface this has been extended in various directions by Hein. Hein's results apply to any K3 surface, but we only state them in the case of Picard number one.

\begin{theorem}[Hein]\label{hein}
    Let $(S,\O_S(1))$ be a smooth projective K3 surface of degree $2t$ and Picard number one.
    \begin{enumerate}[$(1)$]
        \item  If $t(2m-1)>48$ then for any smooth curve $C\in|\O_S(m)|$  we have that $\Omega^1_S|_C$ is a stable vector bundle. 
        \item  For any $t\geq1$, $m\geq1$, $(t,m)\neq(1,1),(1,2)$ and general curve $C\in|\O_S(m)|$, the restriction $\Omega^1_S|_C$ is semistable.
    \end{enumerate}
\end{theorem}
\begin{proof}
    The first statement follows directly from \cite[Theorem 2.8]{hein}. Then (ii) follows from \cite[Korollar 3.11]{heinthesis}, and the fact that a general member of $|\O_S(m)|$ is not hyperelliptic for these values of $t$ and $m$.
\end{proof}
To be more concrete, for $m=1$ we obtain $t>48$, for $m=2$, $t>16$ and $m=3$ gives $t>9$. On the other hand, for $t\geq1$ we obtain that $m\geq25$, which is an improvement on the bound of Bogomolov mentioned above.

In other words finding curves which destabilise the cotangent bundle is limited to low degree K3s and low degree multiples of the polarisation, and even more so only to small parts of the linear systems considered. 
Constructing such curves in general seems out of reach, so in Section \ref{sectionexamples} we proceed on a case by case basis.  

An initial natural attempt is to consider the ramification curves of a generic projection $S\to \PP^2$.  The starting point here is the fact that in the degree two case, the ramification curve $C$ of the degree two cover $f : S \to \PP^2$ is smooth, and the restriction of $\Omega^1_S$ to it is not semistable (see Section \ref{sectiondeg2}). As explained in the proof below, these curves lie in $|\O_S(3)|$, so Hein's Theorem \ref{hein} implies that they cannot be destabilising if $t>9$.

\begin{proposition}\label{propramdiv}
    Let $(S,\O_S(1))$ be a K3 surface of degree $2t>2$ and $S\subset\PP^{t+1}$ the induced embedding. For a linear space $\Lambda=\PP^{t-2}\subset\PP^{t+1}$ denote by $R_\Lambda\subset S$ the ramification divisor of the projection $S\to \PP^2$ from $\Lambda$. If $\Lambda$ is general and $t\leq3$ then $\Omega^1_S|_{R_\Lambda}$ is not semistable, whereas for $t=4$ it is strictly semistable.
\end{proposition}
\begin{proof}
From \cite{cilibertoflamini} we know that since $\Lambda$ is general, the projection $S\to\PP^2$ is a morphism and $R_\Lambda$ is a smooth irreducible curve in $|\O_S(3)|$. Consider the natural morphism $\phi:\PP(\Omega_S^1)\to\operatorname{Gr}(2,t+2)$ (see Section \ref{sectioncones}). This factors as
$$\PP(\Omega^1_S)\subset\PP(\Omega^1_{\PP^{t+1}}|_S)\subset\PP(\Omega^1_{\PP^{t+1}}) = \PP(U^\vee) \xrightarrow{p}
    \operatorname{Gr}(2,t+2)$$where $U$ is the universal subbundle on the Grassmannian and $p$ is the bundle projection. Let us consider $\sigma_2\subset \operatorname{Gr}(2,t+2)$ the Schubert cycle parameterising lines meeting the $(t-2)$-plane $\Lambda$. The class $\phi^{*}\sigma_2$ is represented by a smooth curve $C\subset \PP(\Omega_S^1)$ whose image in $S$ is a curve whose points correspond to lines meeting $S$ with multiplicity at least two that also meet $\Lambda$, i.e. the curve $R_\Lambda$. We claim that $L$ is negative on this curve. 

To compute the class of $C$, note that $\sigma_2=c_1^2(U^\vee)-c_2(U^\vee)$. On $\PP(U^\vee)$, the line bundle $\O_{\PP(U^\vee)}(1)$ corresponds to $H$ and $g=L+2H$ corresponds to the pullback of the Pl\"ucker polarisation. Using the Grothendieck relation on $\PP(U^\vee)$ we find that $\phi^*\sigma_2=L^2+3L\cdot H+3H^2$. This gives $L\cdot\phi^*\sigma_2 = 6t-24$ which is negative if $t\leq3$, so $L\cdot C<0$. Since $R_\Lambda$ is smooth, we have from Lemma \ref{lazlemma} that $\Omega^1_S|_{R_\Lambda}$ is not semistable for $t\le 3$. If $t=4$, we get a degree zero quotient of the vector bundle $\pi^*\Omega^1_S$ restricted to the curve $C$, and hence by pushing down a degree zero quotient on $R_\Lambda$. Hence $\Omega^1_S|_{R_\Lambda}$ is semistable as it is an extension of degree zero line bundles, but not stable.
\end{proof}

\subsection{Results of Bayer--Macr\`i}\label{BMsection}
To approach the computation of the pseudoeffective and nef cones of $\PP(\Omega^1_S)$, we will repeatedly be using the results of Bayer--Macr\`i \cite[Section 13]{bayermacri}. Note that the pseudoeffective and movable cones are dual under the Beauville--Bogomolov--Fujiki form, which allows the computation of the former from the results \emph{loc. cit.} on the latter (see \cite[p.570]{bayermacri}).

\begin{theorem}[Theorem 13.1 and 13.3 in \cite{bayermacri}]\label{bm}
    Let $(S, \O_S(1))$ be a polarised K3 surface of degree $2t$ and Picard number one.
    \begin{enumerate}[$(1)$]
        \item Assume the equation $x^2-4ty^2=5$ has no solutions.
        \begin{enumerate}[$(a)$]
            \item If $t$ is a perfect square then $\nef(S^{[2]}) = \mov(S^{[2]}) = \langle \tH, \tH - \sqrt{t}B
            \rangle$, $\eff(S^{[2]})=\langle B, \tH-\sqrt{t}B\rangle$.
            \item If $t$ is not a perfect square then $x^2-ty^2=1$ has a minimal solution\footnote{Minimal
            meaning $a$ is minimal and $a,b>0$.} $(a,b)$ and the cones are given by $\nef(S^{[2]}) = \mov(S^{[2]}) = \langle \tH, \tH -
            t\frac{b}{a}B\rangle$ and $\eff(S^{[2]})=\langle B, \tH-\frac{a}{b}B\rangle$.
        \end{enumerate}
        \item If the equation $x^2-4ty^2=5$ has a minimal solution $(c,d)$ then\footnote{Note there is a typo in \cite[13.3(b)]{bayermacri}, fixed here.} $\nef(S^{[2]}) = \langle \tH,
        \tH-2t\frac{d}{c}B\rangle$ whereas $\mov(S^{[2]})$ and $\eff(S^{[2]})$ are as in Case $(1)$.
    \end{enumerate}
\end{theorem}
The above results in fact give a complete picture of the birational models of $S^{[2]}$. We give a short description
here as they will be important in what follows. First note that one extremal ray of the movable cone of $S^{[2]}$ is
given by $\tH$, and it induces the Hilbert--Chow contraction to $\operatorname{Sym}^2S$. We now describe the other ray of the movable
cone, which we will often call \textit{the second extremal ray}. Firstly, assume that the Pell-type equation 
$$x^2-4ty^2=5$$
does not have a solution. If $t$ is a square we have a Lagrangian fibration $S^{[2]}\to\PP^2$, and hence the second
boundaries of $\nef(S^{[2]})$ and $\eff(S^{[2]})$ agree. If $t$ is not a square then (see \cite[Proposition
3.6]{debarremacri}) there is a divisorial contraction
$S^{[2]}\to Y$ contracting an irreducible divisor $D$ to a smooth K3 surface $T\subset Y$ and the restriction to $D$ is a
$\PP^1$-fibration $D\to T$ (in the sense that the general fibre is $\PP^1$). We note that like in the example of the
Hilbert--Chow contraction to $\operatorname{Sym}^2 S$, the variety $Y$ need not be smooth. 

If on the other hand the above Pell-type equation does have a solution, then there is a finite sequence of Mukai flops
to a smooth irreducible holomorphic symplectic variety $S^{[2]}\dashrightarrow X^+$ which, as above, either admits a
divisorial contraction $X^+\to Y$, contracting an irreducible divisor $D\to T$ to a smooth K3 surface $T$ (if $t$ is not a
square), or has a Lagrangian fibration $X^+\to\PP^2$ (if $t$ is a square).

\section{The cones of divisors of \texorpdfstring{$\PP(\Omega^1_S)$}{Lg}}\label{sectioncones}\noindent We begin with
some simple bounds for the nef and pseudoeffective cones on $\PP(\Omega^1_X)$ for a smooth projective variety $X\subset
\PP^n$. A convenient ingredient here is the morphism \begin{equation}\label{Grmap}f:\PP(\Omega^1_{X})\to {\rm
Gr}(2,n+1)\end{equation} which associates a tangent vector to the corresponding point in the Grassmannian. For $L$ the
tautological line bundle on $\PP(\Omega^1_X)$ and $H$ the polarisation on $X$ as in Section \ref{prelim}, one checks
that $f^*\O_{{\rm Gr}(2,n+1)}(1)=L+2H$. When $f$ maps to a lower-dimensional variety, $L+2H$ is extremal in the nef cone
as well as the pseudoeffective cone.

\begin{proposition}\label{coneprojspace}
  The vector bundle $\Omega^1_{\PP^n}(m)$ is nef if and only if $m\geq 2$. More precisely, 
  $$\nef(\PP(\Omega^1_{\PP^n})) = \eff(\PP(\Omega^1_{\PP^n})) = \left\langle H, L+2H\right\rangle.$$Moreover, if $X\subseteq \PP^n$ is a smooth subvariety, then $L+2H$ is nef, and ample if and only if $X$ does not contain a line.\end{proposition}
\begin{proof}The equality of the cones follows from the preceding paragraph. The surjection $\Omega^1_{\PP^n}|_X\to \Omega^1_X$ allows us to factor $f$ as $\PP(\Omega^1_X)\subset \PP(\Omega^1_{\PP^n}|_X)\subset \PP(\Omega^1_{\PP^n})\to \operatorname{Gr}(2,n+1)$. So $L+2H$ is base-point free, and ample if and only if $f$ is finite, i.e., when $X$ does not contain a line.
\end{proof}

Thus, if $X\subset\PP^n$ a smooth subvariety, the slope of the nef cone of $\PP(\Omega^1_X)$ (with respect to $\O_{\PP^n}(1)|_X$) is at most 2. Using results of Section \ref{BMsection}, we can improve this in the case of K3 surfaces. Note first that if $S$ a K3 surface of degree $2t$, then the nef cone is contained in the positive cone $\{D \,| \, D^3\ge 0\}$. However, the boundary divisor $L+\frac{2}{\sqrt{t}}H$ of this cone can in general fail to be nef (see Section 4). Nevertheless, the following shows these two cones approximate each other if $t$ is large. 
\begin{proposition}\label{nefbound}
Let $(S, \O_S(1))$ be a K3 surface of degree $d=2t$ and Picard number one. Then for $t>1$ we have 
\begin{equation*}
\alpha_{n,2t} \leq \frac{2}{\sqrt{t-\frac{5}{4}}}\,\text{ and }\, \alpha_{e,2t} \leq \frac{2}{\sqrt{t}}. 
\end{equation*}
There are similar lower bounds for $\alpha_{n,2t},\alpha_{e,2t}$, showing that they both limit to $\frac{2}{\sqrt{t}}$ for large $t$.
\end{proposition}
\begin{proof}
   From Theorem \ref{bm} the nef cone of $S^{[2]}$ is determined by the fundamental solution to the Pell-type equation $x^2-4ty^2=5$. Assume first that this has no solutions. In this case the nef and movable cones agree. First, if $t$ is a square then the divisor $\tH-\sqrt{t}B$ is the boundary of the nef cone. If $t$ is not a square, for a solution $(x,y)$ of $x^2-ty^2=1$, we must have $x \geq \sqrt{t}$. Rewriting the equation as $$\frac{y}{x}=\frac{1}{\sqrt{t}}\sqrt{\frac{x^2-1}{x^2}}$$shows that
$\frac{y}{x}\ge \frac{\sqrt{t-1}}{t}$, and so we find  that the divisor $\tH-\sqrt{t-1}B$ is nef. If $x^2-4ty^2=5$ has a solution, then for such a solution $(x_1,y_1)$ with $x_1>0$ minimal and $y_1>0$, the boundary of the nef cone is given by $\tH-2t\frac{y_1}{x_1}B$. It must be that $x_1\geq2\sqrt{t}$ from which the Pell-type equation gives $$\frac{y_1}{x_1}\geq \frac{\sqrt{4t-5}}{4t}.$$ This implies that $\tH-aB$ is nef for $a=\frac{\sqrt{4t-5}}{2}$. From this we see that regardless of whether $x^2-4ty^2=5$ has a solution or not, the divisor $(\tH-\sqrt{t-\frac{5}{4}}B)\big|_E=2H+\sqrt{t-\frac{5}{4}}L$ is nef on $E$ which implies the result.

	For the effective slope, if $t$ is a square then from Theorem \ref{bm} we have that $D=\tH-\sqrt{t}B$ is nef and extremal on $S^{[2]}$, and since $D^3=0$, the same is true for the restriction $L+\frac{2}{\sqrt{t}}H$ on $E$. If $t$ is not a square, then working as above we see that any minimal solution $(a,b)$ to this equation must satisfy $\frac{b}{a}<\frac{1}{\sqrt{t}}$, which means that $\tH-\sqrt{t}B$ is pseudoeffective on $S^{[2]}$ and hence so is its restriction $L+\frac{2}{\sqrt{t}}H$.
    
    The lower bound from $\alpha_{e,2t}$ comes from the fact that an effective divisor $D$ must satisfy the inequality $D(\sqrt{t-\frac{5}{4}}L+2H)^2\ge 0$. This gives $\frac{8 t-15}{2 t \sqrt{4 t-5}}<\alpha_{e,2t}$, and this expression is asymptotic to $\frac{2}{\sqrt t}$. There is a similar lower bound for $\alpha_{n,2t}$ obtained by squaring a movable class on $S^{[2]}$.
\end{proof}
We will see below that the bound for $\alpha_{e,2t}$ is actually attained for infinitely many $d$. First, we prove the following result for the nef cone:

\begin{theorem}\label{theonefcone}
    Let $(S,\O_S(1))$ be a K3 surface of degree $2t$  and Picard number one. Assume that either $t$ is a square or that $x^2-4ty^2=5$ has a solution. Then the restriction map
    \begin{equation}\label{restrictionnef}
    \nef(S^{[2]})\to \nef(E)
    \end{equation}is an isomorphism.
\end{theorem}
\begin{proof}
Dually, it suffices to prove that the inclusion map $i:E\to S^{[2]}$ induces a surjective map on the cones of curves
    $i_*:\operatorname{NE}(E)\to \operatorname{NE}(S^{[2]})$. In the Picard number one case,  $\operatorname{NE}(S^{[2]})$ is generated by
    a fibre of the Hilbert--Chow contraction (which is already contained in $E$) and some other extremal class $l$. We use
    the description of the nef cone of $S^{[2]}$ due to Bayer--Macr\`i to show that $l=i_*R$ for some effective 1-cycle  $R\in
    \operatorname{NE}(E)$. From \cite[Proposition 3.6]{debarremacri} there are three cases to consider
    \begin{enumerate}[(i)]
        \item There is a primitive integral nef class $D$ with square zero with respect to the Beauville--Bogomolov--Fujiki form. Here $l$ corresponds to a
        curve in one of the fibres of the associated Lagrangian fibration $S^{[2]}\to \PP^2$.
        \item There is a Mukai flop $S^{[2]}\dashrightarrow X^+$ and the class $l$ corresponds to a line in the flopped
        projective plane $P=\PP^2$.
        \item There is a divisorial contraction $\pi:S^{[2]}\to Z$ contracting a divisor $D$, and $l$ corresponds to a fibre
        of $\pi|_D:D\to Y$.
    \end{enumerate}  The conditions on $t$ imply in particular that we are not in case (iii), so we consider the other two. In case (i) we have $D|_E^3=0$ and so $D$ is also extremal on $E$, meaning the restriction map
    \eqref{restrictionnef} is indeed an isomorphism. 
    
    In case (ii), note first that $P\cap E\neq \emptyset$: If not then the class of a line would have to also satisfy $\tH \cdot \ell=0$, but this does not happen though as the Hilbert--Chow morphism does not contract $\ell$.  This means that $P$ and $E$ intersect along a curve $R$, and $\ell$ equals a multiple of $i_*R$. 
\end{proof}

\begin{remark} \label{quarticremark}
    In case (iii), the map \eqref{restrictionnef} is not always an isomorphism. This is because the divisorial
    contraction $S^{[2]}\to Z$ may restrict to a finite map on $E$. This happens for instance when the degree is four or six (see Sections \ref{sectiondeg4} and \ref{sectiondeg6}).
\end{remark}

We now turn to the problem of computing the pseudoeffective cone of $\PP(\Omega^1_S)$ of a K3 surface $S$ in some
special situations. First we observe that if the Hilbert scheme admits a Lagrangian fibration $S^{[2]}\to\PP^2$, i.e.,
when $t\ge 2$ is a square and Equation \eqref{Pell} has no solution, we have that the second extremal divisors of
$\eff(S^{[2]})$ and $\nef(S^{[2]})$ agree, and since $E$ is mapped to a lower-dimensional variety, this extremal divisor
$\tH-\sqrt{t}B$ restricts to an extremal ray of both $\eff(E)$ and $\nef(E)$. In particular, the above bound of
Proposition \ref{nefbound} is optimal for infinitely many values of $d$.

\begin{corollary}\label{proplagrangian}
    Let $(S,\O_S(1))$ be a K3 surface of degree $2t$ and assume that $t$ is a square and that $x^2-4ty^2=5$ has no solutions. Then
    $$ \eff(\PP(\Omega^1_S)) = \nef(\PP(\Omega^1_S))=\left\langle H, L+\frac{2}{\sqrt{t}}H\right\rangle. $$
\end{corollary}
\noindent We remark, however that $\nef(E)$ need not equal $\langle H, L+\frac{2}{\sqrt{t}}B\rangle$ in general - this happens for instance for quartic surfaces (see below).

We will now focus our attention on K3 surfaces $S$ of Picard number one and degree $d=2t$ so that $S^{[2]}$ has two
divisorial contractions, namely the Hilbert--Chow and a contraction to a normal variety $Y$. We denote the extremal divisor of this contraction by $D\subset S^{[2]}$. As pointed out in an earlier section, the image of $D$ via the contraction is a smooth K3 surface $T\subset Y$. 

\begin{proposition}\label{theoinvolution}
    Let $(S,\O_S(1)$ be a K3 surface of degree $2t$ and Picard number one. Assume that $t$ is not a square and
    $x^2-4ty^2=5$ has no solutions\footnote{Equivalently $S^{[2]}$ has two divisorial contractions, see the discussion
    after Theorem \ref{bm}.}. In the notation above, assume furthermore that the contracted divisor is
    $D\simeq\PP(\Omega^1_T)$ for $T$ a K3 surface. Then $$\eff(E) = \langle D|_E, H\rangle.$$
\end{proposition}
\begin{proof}
    We want to show that $D|_E$ is extremal in the pseudoeffective cone, or equivalently that $h^0(E, mD|_E)$ grows at
    most quadratically in $m$. First note that $h^0(S^{[2]},mD)=1$, since $D$ is an exceptional divisor. From the exact sequence
    $$0\to \O_{S^{[2]}}(mD-E)\to \O_{S^{[2]}}(mD)\to \O_E(mD)\to 0$$
    we see that it suffices to show that $h^1(S^{[2]},mD-E)$ is bounded as a function of $m$, namely that for
    $m$ large enough the dimension of this vector space stabilises; this will imply that $mD|_E$ can only have one
    section. For this, we consider the sequence
    $$0\to \O_{S^{[2]}}((m-1)D-E)\to \O_{S^{[2]}}(mD-E)\to \O_D(mD-E)\to 0.$$
    It suffices to prove that $\HH^1(D,\O_D(mD-E))=0$ for $m$ large, because then the maps
    $$\HH^1(S^{[2]},\O_{S^{[2]}}((m-1)D-E))\longrightarrow \HH^1(S^{[2]},\O_{S^{[2]}}(mD-E))$$ are eventually surjective, and so the dimensions 
    $h^1(S^{[2]},mD-E)$ must stabilise. Let $p:D\to T$ the
    projection and write $E|_D=\O_{\PP(\Omega^1_T)}(a)\otimes p^*\O_T(b)$ for positive integers $a,b$. Note that by adjunction we have $\O_D(D)\simeq \omega_D=O_{\PP(\Omega_T)}(-2)$. Using the Leray spectral sequence we now find
    \begin{eqnarray*}
        \HH^1(D,\O_D(mD-E))&=&\HH^1(\PP(\Omega^1_T),\O_{\PP(\Omega^1_T)}(-(2m+a))\otimes p^*\O_T(-b)) \\
        &=&\HH^0(T,(\Sym^{2m+a-2}\Omega^1_T)\otimes \O_T(-b)),
    \end{eqnarray*}
and this is zero for $m\gg 0$, since $\Omega^1_T$ is not pseudoeffective.
\end{proof}
\begin{remark}
    The above method of proof cannot work in general as Hassett and Macr\`i explained to us that in degree 16, the second extremal divisor is not even isomorphic to a projective bundle.
\end{remark}

A classical case where the assumptions are satisfied is when $S^{[2]}$ admits an involution (which is not induced from an involution of $S$); this happens for instance when $S$ is a quartic surface (see Section \ref{sectiondeg4} for details). In fact, the main theorem of \cite{bcnws}, gives a complete classification of the degrees for which $S^{[2]}$ admits an involution which acts non-trivially on the cones of divisors. In this case the second extremal effective divisor $D$ of $S^{[2]}$ is simply the image of $E$ under the involution.

More generally, the proposition applies in the case of ambiguous Hilbert schemes of Hassett (see \cite[Proposition 6.2.2]{hassett:2000}). Here we say that $S^{[2]}$ is \textit{ambiguous} if there exists a smooth K3 surface $T$ so that either $T$ is not isomorphic to $S$ yet has 
$$T^{[2]}\simeq S^{[2]},$$
or $T$ is isomorphic to $S$ but there exists an isomorphism $S^{[2]}\simeq T^{[2]}$ not induced by one between $T$ and $S$. In either case we have $D\simeq\PP(\Omega^1_T)$. In fact there exists a classification in \cite[Proposition 3.14]{debarremacri} or
\cite[Theorem 7.3]{zuffetti}, depending only on $d$, giving when $S$ is ambiguous - namely for $(S,\O_S(1))$ a K3 surface of Picard number one and degree $2t$, then $S^{[2]}$ is ambiguous if and only if the three listed conditions of the following theorem are fulfilled. To summarise one obtains the following.

\begin{theorem}\label{effconecor}
    Let $(S,\O_S(1))$ be a K3 surface of degree $2t$ and Picard number one. Assume the following three conditions
    hold 
    \begin{enumerate}[$(1)$]
        \item $t$ is not a square, 
        \item $x^2-4ty^2=5$ has no solutions, 
        \item the minimal solution $(a,b)$ of $x^2-ty^2=1$ has $b$ even.
    \end{enumerate}
    Then $\eff(E)=\langle H, D|_E\rangle$ where $D$ the second extremal divisor from $\eff(S^{[2]})$.
\end{theorem}

\section{Examples in low degrees}\label{sectionexamples}

We work through the first four cases of degree $d=2t$ Picard number one K3 surfaces, in each case applying and extending the results of the previous section and explain the underlying geometry of the numbers. Note that some values of $\alpha_{e,d},\alpha_{n,d}$ of low degree K3s were computed in \cite{oguisopeternell}, although their results are incorrect in the degree 4 case. We correct this below and study the problem for K3 surfaces of Picard number one of degree $d\leq8$. We summarise the results in the following table.

\vspace{10pt}\begin{center}
\renewcommand{\arraystretch}{2}
\begin{tabular}{c|c|c}

$d$ & $\alpha_{e,d}$ & $\alpha_{n,d}$ \\
\hline
2 & $\le 2$  & 3 \\
\hline
4 & $\frac{4}{3}$ & $\sqrt{2}<\alpha_{n,4}<\frac{3}{2}$ \\
\hline
6 & 1 & $\frac{6}{5}<\alpha_{n,6}<\frac{4}{3}$ \\
\hline
8 & 1 & 1
\end{tabular}
\vspace{10pt}\captionof{table}{$\eff(\PP(\Omega^1_S))=\langle H, L+\alpha_{e,d}H\rangle, \nef(\PP(\Omega^1_S))=\langle H,
L+\alpha_{n,d}H\rangle$}
\renewcommand{\arraystretch}{1}
\end{center}

In addition, following Section \ref{sectiondestabcurves}, we give examples of positive dimensional families of smooth curves in $S$ so that $\Omega^1_S|_C$ is not
semistable on the general member.

\subsection{$d=2$, Double covers of $\PP^2$}\label{sectiondeg2}
Let $S$ be a generic degree two K3 surface, and let $f:S\to \PP^2$ be the double cover with ramification divisor $R$ and
branch locus a plane sextic curve $C$. There is an exact sequence
\begin{equation}\label{deg2ramificationseqn} 
0\to f^*\Omega^1_{\PP^2}\to \Omega^1_S\to \Omega^1_f\to 0.
\end{equation}
Here $f$ restricts to an isomorphism on $R$, which from the standard cotangent sequence corresponding to the composition
$R\subset S\to\PP^2$ implies that $\Omega^1_f|_R = \O_S(-R)|_R=\O_R(-3)$. In particular after
restricting to $R$ we obtain a negative quotient $\Omega^1_S|_R\to \O_R(-3)\to 0$. This gives not only a smooth curve in $S$ which destabilises the cotangent bundle, but also a curve $C\subset
E=\PP(\Omega^1_S)$ which is extremal in the cone of curves. Indeed $C$ is exactly the intersection of the
$\PP^2\subset S^{[2]}$ (given by the preimages of $S\to\PP^2$) with $E$, and so it can be contracted to a point.

From Theorem \ref{bm} we have that $\nef(S^{[2]})=\langle\tH, \tH-\frac{2}{3} B\rangle$,
$\mov(S^{[2]})=\eff(S^{[2]})=\langle \tH, \tH-B\rangle$ and that the second birational model of $S^{[2]}$ in this case is a rational Lagrangian fibration, i.e.,\ a Lagrangian fibration after a single flop of the aforementioned $\PP^2$ (see \cite[Example 9(i)]{bakker} for further description of the geometry). From Theorem \ref{theonefcone} one computes that the divisor $(\tH-\frac{2}{3} B)|_E=\frac{2}{3} L+2H$ is extremal in the nef cone of $E$ and so $\alpha_{n,2}=3$. Similarly, the restriction $(\tH-B)|_E=L+2H$ is effective on $E$ (and in fact it is movable).

We come now to movable curves destabilising the cotangent bundle. We thank Roberto Vacca for spotting a mistake in a previous construction in the following.
    
\begin{proposition}\label{propdeg2destab}
    Let $S$ be a generic K3 surface of degree two. Then there is a family of smooth curves in $|\O_S(6)|$ destabilising the cotangent bundle. They are the images of curves of class $(L+2H)(L+4H)$ under the projection $\PP(\Omega^1_S)\to S$.
\end{proposition}
\begin{proof}
To avoid repeating the same argument we refer to the case of the quartic below in Proposition \ref{propdeg4destab} for what follows. 

From Sequence \eqref{deg2ramificationseqn}, since $\Omega^1_f$ is supported on $R$ and $\Omega_{\PP^2}(2)$ is globally generated, the divisor $D=L+2H$ on $\PP(\Omega^1_S)$ has base locus supported on the curve $R_S=\PP(\Omega_f^1)\subset\PP(\Omega^1_S)$ (isomorphic to $R$) and three linearly independent sections. 

The lift $t:R\to\PP(\Omega^1_S(2H))$ induced by the quotient $\Omega^1_S(2H)\to\Omega^1_f(2H)\cong\O_R(2H-R)\to0$ satisfies $t^*\O(1)\cong\O_R(2H-R)\cong\O_R(-H)$ since $\O_R(-R)=\O_R(-3H)$. Since this divisor is indivisible in $\Pic(R)$, we conclude that the base locus of $L+2H$ is in fact reduced and hence isomorphic to the smooth curve $R$. By the Strong Bertini Theorem \cite[Proposition 5.6]{3264}, we conclude that general sections of $L+2H$are smooth.

A general $S'\in|D|$ is thus smooth and birational to $S$, and by the adjunction formula
$K_{S'}\cong2H-L$, we get that $K_{S'}^2=-32$, so in particular $S'$ is the blowup of $S$ in 32
points - as $S'\subset\PP(\Omega^1_S)$ it can only have irreducible exceptional divisors.
Furthermore, $L+4H$ is ample and base point free (again by sequence \ref{propdeg4destab}). This
implies that there are smooth movable curves $C\in |(L+4H)|_{S'}|$. As $K_{S'}\cdot (L+4H)=32$, the
image of $C$ in $S$ is smooth and $C$ maps isomorphically onto its image.

Now $\deg L|_C=L\cdot (L+2H)\cdot (L+4H)=-8$, which means the quotient line bundle
$$\pi^*\Omega^1_S|_C \to L|_C\to 0$$
is negative on $C$.  Hence $\Omega^1_S$ has a negative quotient when restricted to the image of $C$ in $S$ and so it is not nef. Note that $(L+2H)(L+4H)\cdot H=12$, so that the images of these curves move in a positive dimensional family in $|\O_S(6)|$ on $S$. As $h^0(L+2H)=3$ and $h^0(L+4H)=12$, we see that these curves form a 15-dimensional subvariety in $|O_S(6)|$.

A similar argument shows that general sections of $L+3H$ are smooth, so that one can also produce large families of smooth curves coming from $(L+3H)(L+4H)$ whose image in $S$ is smooth and the restriction of $\Omega^1_S$ is strictly semistable since $L(L+3H)(L+4H)=0$.
\end{proof}

\subsection{$d=4$, Quartics in $\PP^3$}\label{sectiondeg4} Let $S\subset\PP^3$ be a quartic surface of Picard number one. As studied by Beauville in \cite{beauville}, $S^{[2]}$ has an
involution $i:S^{[2]}\to S^{[2]}$ which maps the reduced length two subscheme $p+q\in S^{[2]}$ to the residual degree
two subscheme on the intersection $S\cap\ell_{p,q}$ of $S$ with the line through $p$ and $q$. When $S$ has Picard number one, this involution must swap the boundaries of the pseudoeffective and nef cones. 

There is a natural morphism $f:S^{[2]}\to {\rm Gr}(2,4)$ taking a $0$-cycle of degree two and associating to it the line in $\PP^3$ spanned by the two
points. Since a general line meets $S$ in four points, this is generically finite of degree $\binom{4}{2}=6$. In terms of $\tH,B$ the Pl\"ucker polarisation pulls back to $\tH-B$. In particular
$\tH-B$ is fixed by the involution $i$. From this, one computes that $i(\tH)+\tH=4(\tH-B)$, and hence
$i(\tH)=3\tH-4B$ and $i(B)=2\tH-3B$. These two divisors are clearly extremal in $\eff(S^{[2]})$ and $\nef(S^{[2]})$ respectively, so we get $\nef(S^{[2]})=\mov(S^{[2]})=\langle\tH, 3\tH-4
B\rangle$ and $\eff(S^{[2]})=\langle B, 2\tH-3 B\rangle$.

\subsubsection{The pseudoeffective cone of $\PP(\Omega^1_S)$} This has been computed in Theorem \ref{effconecor}. Since the
second boundary of the pseudoeffective cone of $S^{[2]}$ restricted to $E$ is given by $(\tH-\frac{3}{2}B)|_E=\frac{3}{2}L+2H$ we
obtain the following.

\begin{corollary}\label{quarticcor}
    Let $S\subset\PP^3$ be a smooth quartic surface of Picard number one. Then for $L=\O_{\PP(\Omega^1_S)}(1)$ and $H$
    the pullback of $\O_S(1)$ to $\PP(\Omega^1_S)$ we have $$\eff(\PP(\Omega^1_S)) = \left\langle 3L+4H, H\right\rangle.$$
\end{corollary}\noindent The divisor $i(E)|_E=6L+8H$ is effective on $E$. Indeed, this is exactly the surface of bitangents $U_0\subset \PP(\Omega^1_S)$ (see \cite[Proposition 2.3]{tihomirov} and
\cite[Proposition 3.14]{welters}). On the other hand, the generator $3L+4H$ is not itself effective: taking symmetric powers of the standard cotangent sequence associated to a quartic in $\PP^3$ one obtains
$$0\to\Sym^{k-1}\Omega^1_{\PP^3}|_S\otimes \O_S(-4)\to \Sym^k\Omega^1_{\PP^3}|_S\to \Sym^k\Omega^1_S\to 0.$$
We can compute the cohomology of the middle term using the sequences
\begin{eqnarray*}
0\to\Sym^k\Omega^1_{\PP^3}(m-4)\to\Sym^k\Omega^1_{\PP^3}(m)\to\Sym^k\Omega^1_{\PP^3}|_S(m)\to 0 \\
0\to\Sym^k\Omega^1_{\PP^3}\to \O_{\PP^3}(-k)^{\oplus \alpha_k}\to \O_{\PP^3}(-k+1)^{\oplus\beta_k}\to 0\label{symeuler}
\end{eqnarray*}
where $\alpha_k=\binom{k+n}{k}$ and $\beta_k=\binom{k+n-1}{k-1}$. Hence $\HH^0(\PP(\Omega^1_S),3L+4H)=0$, since it is sandwiched between $\HH^0(S, \Sym^3\Omega^1_{\PP^3}|_S(4))$ and $\HH^1(S, \Sym^2\Omega^1_{\PP^3}|_S)$, the
first of which is zero (e.g., by Proposition \ref{coneprojspace}), whereas the second is zero by Kodaira vanishing.

\subsubsection{The Nef cone of $\PP(\Omega^1_S)$} 
Going on now to explain Remark \ref{quarticremark}, consider again the surface $U_0=i(E)\cap E$ of bitangents of $S$.
The Hilbert--Chow morphism restricts to a finite morphism $U_0\to S$ of degree six (see \cite{welters}; counted with
multiplicities, there are six bitangent lines through any point in $S$). In other words, the morphism $i(E) \to
\operatorname{Sym}^2(S)$, induced by $i(H)$ is finite (since any curve contracted would have to lie in $i(E)\cap E$). This implies that the divisor $i(H)|_E=4L+6H$ is ample on $E$, and hence $\alpha_{n,4}<\frac{3}{2}$.

Now, from the computation of the
pseudoeffective cone above one sees that $L+\sqrt{2}H$ is pseudoeffective, but $(L+\sqrt{2}H)^2(3L+4H)<0$ which means that $L+\sqrt{2}H$ is not nef. Hence $\sqrt{2}<\alpha_{n,4}<\frac{3}{2}$.

\subsubsection{Families of Destabilising Curves} We will give two families of destabilising curves on $S$,
the second of which has a natural geometric description as the family of ramification divisors from a generic projection $S \to \PP^2$.

\begin{proposition}\label{propdeg4destab}
    Let $S\subset\PP^3$ be a generic quartic K3 surface. 
    \begin{enumerate}[$(1)$]
        \item There is an $8$-dimensional family of smooth curves in $|\O_S(4)|$ destabilising the cotangent bundle. These
        are the images of the smooth curves of class $(L+2H)^2$ in $\PP(\Omega^1_S)$, under the projection to $S$.
        \item There is a 3-dimensional family of smooth curves in $|\O_S(3)|$ destabilising the cotangent bundle.
        Geometrically this is the family of ramification divisors $R_p$ of the projection morphism $f_p:S\to\PP^2$ from a
        generic point $p$ and the destabilising quotient is given by $\Omega_{f_p}|_{R_p}$.
    \end{enumerate}
\end{proposition}
\begin{proof}
Observe first that the line bundle $D=L+2H$ is base-point free on $\PP(\Omega^1_S)$ by Proposition \ref{coneprojspace}. By Bertini, there is a smooth surface  $S'\in |D|$. Recall that on $\PP(\Omega^1_S)$, we have $L^3=-24$, $L^2\cdot H=0$, $L\cdot H^2=4$. Let $\ell=L|_{S'}$ and $h=H|_{S'}=f^*\O_S(1)$. This implies that we have the following intersection numbers for $\ell$ and $h$ on $S'$: $\ell^2=-24$, $\ell\cdot h=8$ and $h^2=4$. It follows that $\ell\cdot(\ell+2h)=-8$. Note that $D|_{S'}=\ell+2h$ is nef, so the restriction of $L$ is not pseudoeffective on $S'$. 

Denote by $f:S'\to S$ the restriction of $\pi$ to $S$; $S'$ is generically a section of $\pi$, so $f$ is birational. In fact, computing the canonical bundle of $S'$ gives $K_{S'}=(K_{\PP(\Omega^1_S)}+D)|_{S'}=2H-L $ which gives that $K_{S'}^2=-40$. In particular $f:S'\to S$ must be the blow-up of $S$ in 40 distinct points since any exceptional curve is contracted to $S$ so must be a fibre of $\PP(\Omega^1_S)\to S$. Moreover $D \cdot K_{S'}=(L+2H)\cdot(-L+2H)=40$ so the general divisor in $C\in |D|_{S'}|$ passes through each exceptional divisor with multiplicity one, implying that its image $C'=f(C)$ in $S$ is smooth. We have $16=D^2\cdot H=\pi_*D^2\cdot\O_S(1)$ and so it follows that $C'$ lies in $|\O_S(4)|$. In fact we obtain an 8-dimensional family this way by intersecting two generic divisors in $|D|$.

The second family was already proven to be destabilising in Proposition \ref{propramdiv}. Using \cite[Proposition 2.6]{pillow} we compute that for $p\in \PP^3$ generic, the ramification divisor $R$ is a smooth genus 19 curve and it is the normalisation of its image, the branch divisor $B$. Also, the branch divisor $B\subset\PP^2$ is irreducible of degree 12 and from \cite{cilibertoflamini} has only cusps and nodes as singularities. Using the results \emph{loc. cit.} one computes it has 12 nodes and 24 cusps, and since only the latter contribute to ramification, this implies that $\deg\Omega^1_{R/B}=24$. The conormal bundle $I/I^2$ of $R\subset S$ is $\O_R(-3)$ and from the sequence 
\begin{equation}\label{relcotangentseqn}
0\to\O_R(-3)\to\Omega^1_f|_R\to\Omega^1_{R/B}\to0
\end{equation}
we see that since the left has degree $-36$ and the right $24$, we obtain that $\deg\Omega^1_f|_R=-12$. From the sequence $\Omega^1_S|_R\to\Omega^1_f|_R\to0$ we thus obtain a negative quotient, which is the desired explicit destabilising quotient.
\end{proof}

\subsection{$d=6$, Intersections of cubics and quadrics in $\PP^4$}\label{sectiondeg6}

Here $\nef(S^{[2]})=\mov(S^{[2]})=\langle\tH, \tH-\frac{3}{2} B\rangle$ and $\eff(S^{[2]})=\langle B, \tH-2B\rangle$. In
this case the second birational model is a divisorial contraction like in the $t=2$ case, but $S^{[2]}$ is not ambiguous, so we are not in the case of Theorem \ref{effconecor}.  

We will compute the effective cone, in particular proving that $\alpha_{e,6}=1$, using the beautiful geometry of $S^{[2]}$ related to nodal cubic fourfolds. We will also show that like in the $d=4$ case, the restriction of the nef cone from $S^{[2]}$ is in the interior of $\nef(\PP(\Omega^1_S))$, obtaining at least in this case that $\alpha_{n,6}<\frac{4}{3}$. 

\begin{remark}
One could alternatively use cohomological techniques to compute that the divisor given by the restriction $2L+2H=(\tH-2B)|_E$ is extremal in the pseudoeffective cone. We do not reproduce the proof here as it is a lengthy computation, but it follows from the fact that $S$ is the intersection $K\cap Q\subset\PP^4$ of a smooth cubic and quadric, so one can take symmetric products of the Euler sequence restricted to $K$ and the cotangent sequence $0\to\O_S(-1)\to\Omega^1_K|_S(1)\to\Omega^1_S(1)\to0$ to show that $h^0(S, \Sym^m(\Omega^1_S(1)))$
does not have cubic growth in $m$.
\end{remark}

We begin by describing a birational model of $S^{[2]}$, as found for example in \cite[Lemma 6.3.1]{hassett:2000}, which is induced by $M=2\tH-3B$. 
Let $$Z=\{x_5q(x_0,\ldots,x_4)+c(x_0,\ldots, x_4)=0\}\subset\PP^5$$ be a general cubic fourfold with a node at $p=[0,0,0,0,0,1]$. Denote by $Q=\{q=0\}, K=\{c=0\}$ the quadric and cubic in $\PP^4$ respectively. Projecting from $p\in\PP^5$ onto this $\PP^4$ we obtain a rational map $Z\dashrightarrow \PP^4$. In the Fano scheme of lines $F$ of $Z$, the locus of lines through $p$ is parameterised by the complete intersection $S=\{q=c=0\}$ in this $\PP^4$ which is a smooth K3 surface of degree six and Picard number one. Consider now the map 
$$f:S^{[2]}\to F$$ taking two lines $\ell_1,\ell_2$ through $p$ and giving the residual line of intersection of $Z$ with the plane spanned by $\ell_1,\ell_2$. Note that this is indeed a morphism (it is not defined precisely when $Z$ contains a 2-plane through $p$, which we may assume does not happen by genericity). Starting with a line $\ell\in F\setminus S$ and considering the plane spanned by $\ell$ and $p$, we obtain a residual singular conic in $Z$ whose constituent two lines give a unique point of $S^{[2]}$. This implies that the map $f$ is birational. Moreover one can prove that $F$ has $A_1$ singularities along $S$ and is smooth otherwise. Finally $S^{[2]}$ is isomorphic to the blowup $\operatorname{Bl}_SF$ of $F$ at $S$ (see \cite[Lemma 6.3.1]{hassett:2000}), and $M$ is the pullback of the Pl\"ucker polarisation. Let $D$ denote the exceptional divisor of $f$; the map $f$ restricts to a smooth $\PP^1$-fibration $D\to S$. Geometrically, this means that if one fixes a line $\ell$ through $p$, then there is a family over $\PP^1$ of 2-planes through $\ell$ whose residual conic is two lines $\ell_1+\ell_2\in S^{[2]}$ passing through $p$. 

Consider now $\Gr(2,4)$ the Grassmannian of lines in $\PP^3$, embedded as a quadric in $\PP^5$ and embed $Q$ from above as a hyperplane section of $\Gr(2,4)$.

\begin{lemma}
    The exceptional divisor $D$ is isomorphic to $\PP(U^\vee|_S)$ where $U$ is the universal bundle on $\Gr(2,4)$. 
\end{lemma}
\begin{proof}
Note that points of $D\subset S^{[2]}$ parameterise pairs of lines through the node $p\in Z$ so that the residual line
in the 2-plane spanned by the two lines also passes through $p$. Projecting from $p$, we see that the images of the
three lines thus lie on a line $\ell\subset Q$, which is a trisecant to $S$.

Let $\mathcal S$ denote the rank 2 spinor bundle on $Q$. We have $\mathcal S=U|_Q$ where $U$ is the universal subbundle
on $\Gr(2,4)$. Also $\mathcal S^\vee\simeq\mathcal S(1)$. The projective bundle $\pi: \PP(\mathcal S^\vee)\to Q$
parameterises lines in $Q$; the fiber over $q\in Q$ is the $\PP^1$ of lines $\ell \subset Q$ through $q$.  We assume
that $S$ contains no lines, so for a point $s\in S\subset Q$, and a line $\ell$ containing $s$, the intersection
$\ell\cap S$ consists of $s$ and a length two subscheme $\xi \in S$. This $\xi$ in turn corresponds to two (possibly
equal) lines $\ell_1,\ell_2$ in $Z$ passing through $p$, and the residual line of these is exactly the line corresponding to
$s$. We have therefore defined a morphism $\PP(\mathcal S^\vee)\to D$. It is clear that this morphism is bijective,
since given $\ell_1,\ell_2$, the line $\ell$ is determined from the secant line through the images of the projections of
$\ell_1,\ell_2$ in $S$.
\end{proof}
\begin{corollary}
    The divisor $L+H$ is extremal in the effective cone, giving $\alpha_{e,6}=1$.
\end{corollary}
\begin{proof}
The line bundle $l=\O_{\PP(U^\vee|_S)}(1)$ is ample since it induces a map $\PP(U^\vee|_S)\to\PP^3$ which is finite.
Indeed, regarding $\PP^3$ as the variety of lines in $Q$, the preimage of a point $[\ell]\in \PP^3$ consists of the
points $\ell \cap S$, which is finite since $S$ contains no lines. Note now that $-D|_D=-K_D$ which equals $2l-g$ for
$g$ the pullback of the hyperplane bundle on $S$. The curve class $l^2$ is ample and one computes
$l^2\cdot(2l-g)=(\sigma_1^2-2\sigma_2)\cdot\sigma_1^2=0$ for $\sigma_i$ the standard Schubert cycles, giving in particular that
$2l-g$ is not pseudoeffective. We will finally sketch how the proof of Proposition \ref{theoinvolution} implies that $D=\widetilde{H}-2B$ is
extremal on $E$, giving that $\alpha_{e,6}=1$. To this aim, we must show that $\HH^1(D,
mD|_D-E|_D)=\HH^1(\PP(U^\vee|_S), mD|_D-E|_D)$ vanishes for $m$ large. This group equals $\HH^0(S,
\Sym^{2m+a}U|_S\otimes(\det U|_S)^{m+b})$, for some $a,b$ integers. It follows from the fact that $2l-g$ is not
pseudoeffective as seen above, along with the isomorphism $U\cong U^\vee\otimes\det U$, that this group is zero for large $m$.
\end{proof}

As far as the nef cone is concerned, we prove the following analogue of the degree four case.

\begin{lemma}
    The restriction of the above morphism $f:D\cap E\to S$ is finite. In particular $D|_E$ is ample.
\end{lemma}
\begin{proof}
If not there exists a line $\ell\subset Z$ through $p$, and a 1-dimensional family $H_t$ of planes intersecting $Z$ in $l$ and a double line $2l_t$. Consider the projection from $p$, $\pi: \PP^5-\{p\}\to \PP^4$, which blows down the lines through $p$ to $S$. The planes $H_t$ give a family of lines in $\PP^4$ passing through a fixed point $q\in S$ and intersecting $S$ with multiplicity two at some other points $q_t$. Since $S=K\cap Q$, we see that $Q$ must contain all of these lines. But the lines of $Q$ passing through $q$ are parameterised by a conic and sweep out a 2-dimensional quadric cone $Q_0$ with vertex at $p$. This quadric $Q_0$ spans a $H\simeq \PP^3$, and the intersection $S\cap H$ contains $S\cap Q_0$ and thus has a non-reduced divisor as a component. However, this is impossible if $\Pic(S)=\ZZ H$.
\end{proof}

In particular, like in the case $t=2$, we have that the restriction $(2\tH-3 B)|_E=3L + 4H$ is ample on $E$ and one sees that the slope of the nef cone
$\alpha_{n,6}$ is strictly less than $\frac{4}{3}$. On the other hand, we compute that $(L+\frac{6}{5} H)^2\cdot(L+H)<0$, so in particular $\frac{4}{3}<\alpha_{n,6}<\frac{6}{5}$.

Finally, we can as before (see Proposition \ref{propramdiv} for a proof) use the map $\phi:\PP(\Omega^1_S)\to {\rm Gr}(2,5)$ to produce an explicit destabilising family of curves on $S$. Indeed, the curve $C\subset \PP(\Omega_S^1)$ of class $\phi^*\sigma_2$ corresponding to lines tangent to $S$ and meeting a general $\PP^1\subset \PP^4$, has intersection number $-6$ with $L$. Pushing such curves forward to $S$ gives a 5-dimensional family of destabilising curves in $|\O_S(3)|$.

\subsection{$d=8$, Intersections of three quadrics in $\PP^5$}\label{sectiondeg8}

Here $\nef(S^{[2]})=\mov(S^{[2]})=\langle\tH, \tH-2B\rangle$ and $\eff(S^{[2]})=\langle B, \tH-2B\rangle$. In this case we have a Lagrangian fibration $S^{[2]}\to\PP^2$ so we know both the cones from Theorem \ref{theonefcone} and Corollary \ref{proplagrangian}. The restriction of the extremal ray of the cones of $S^{[2]}$ to $E$ respectively is $(\tH-2B)_E=2L+2H$ and so we find $\alpha_{e,8}=\alpha_{n,8}=1$. 

The geometry of this fibration is well understood (see \cite[Proposition 7.1]{htlagrangian} and \cite[\S 2.1]{sawon}) and we include some details here for completeness. Consider $S$ as the intersection of three quadrics $Q_i\subset\PP^5$. Then we have a morphism $S^{[2]}\to (\PP^2)^\vee\simeq \PP^2$ defined by taking a $0$-cycle $\xi$ of degree two and giving the base $\PP^1$ of the pencil of quadrics (in the net spanned by the $Q_i$) containing the line in $\PP^5$ spanned by $\xi$. A fibre of this morphism is given by the abelian surface which is the Fano scheme of lines of the threefold which is the base locus of the corresponding pencil. In fact the general fibre will be the Jacobian of a genus two curve: a general pencil of quadrics in $\PP^5$ is singular at precisely six points and we take the genus two curve ramified above these six points. In particular, following the description from Theorem \ref{theonefcone}, it is the class of this curve that spans the Mori cone of $E$, which is also equal to the movable cone of curves of $E$.

\section{Final thoughts and questions}

\subsection{Bott vanishing for K3 surfaces}

The following question was raised by Totaro in \cite{totaro}.
\begin{question}
For which polarised K3 surfaces $(S,\O_S(1))$ does {\em Bott vanishing} hold? That is, for $q>0,p\ge 0$, is
$$
\H^q(S,\Omega^p\otimes \O_S(1))=0?
$$
\end{question}
Here one does not assume that $\O_S(1)$ is primitive. As explained in the above reference, by Kodaira vanishing, one need only treat the case $p=1$.
The problem is essentially solved in \cite{totaro}, but let us show how one can use our above results to easily obtain Bott vanishing for K3 surfaces $(S,\O_S(1))$ of Picard number one with $\O_S(1)^2>75$. Indeed, using the upper bound for the nef slope $\alpha_{n,2t}$ from Proposition \ref{nefbound}, we obtain that if $\O_S(1)^2>75$, the divisor $D=3L+H$ is ample on $E=\PP(\Omega^1_S)$. Since $K_E=-2L$, one finds that $$\H^1(S,\Omega^1_S\otimes \O_S(1))=\H^1(E,L+H)=\H^1(E,K_E+D)=0$$ by Kodaira vanishing.
In general Bott vanishing can fail; for example because $\chi(\Omega^1_S\otimes \O_S(1))$ is negative when $S$ has low degree.

\subsection{Questions}
As noted before, we are only able to compute the cones for certain degrees of K3s, leaving several open questions. For instance, we ask for $E=\PP(\Omega^1_S)$:

\begin{question}
    What is the second extremal ray of $\eff(E)$ for a degree 2 K3 surface, or more generally when $S^{[2]}$ admits a flop?
\end{question}

\begin{question}
    What is the second extremal ray of $\nef(E)$ in degree 4 and 6, or more generally when there is a second
    divisorial contraction $S^{[2]}\to Z$?
\end{question}
\noindent Even though the slopes for $S^{[2]}$ are all rational numbers, the situation for $\PP(\Omega^1_S)$ is less clear:

\begin{question}
    Let $S$ be a K3 surface with Picard number one. Are the cones $\eff(E)$ and $\nef(E)$ rational polyhedral?
\end{question}


\bibliographystyle{alpha}

\end{document}